\documentclass[a4paper,reqno]{amsart}

\usepackage{amssymb}
\usepackage{latexsym}
\usepackage{amsmath}
\usepackage{amsthm}
\usepackage{euscript}
\usepackage{bbm}
\usepackage{tikz}

\def\cA{{\mathcal A}}      
\def\cD{{\mathcal D}}   \def\cE{{\mathcal E}}

\def\cS{{\mathcal S}}      
\def\cV{{\mathcal V}}

\def\cal H{{\mathcal H}}

\def\R{\mathbb{R}}
\def\C{\mathbb{C}}

\def\N{\mathbb{N}}

\def\dom{{\text{\rm dom\,}}}

\def\phi{\varphi}
\def\dd{\textup{d}}

\renewcommand{\theta}{\vartheta}

\newtheorem{theorem}{Theorem}[section]
\newtheorem*{thm*}{Theorem}
\newtheorem{proposition}[theorem]{Proposition}

\theoremstyle{definition}
\newtheorem{definition}[theorem]{Definition}

\newtheorem{remark}[theorem]{Remark}
\newtheorem*{ack}{Acknowledgement}

\numberwithin{equation}{section}

\title{Quantum trees which maximize higher eigenvalues are unbalanced}

\author[J.~Rohleder]{Jonathan Rohleder}
\address{Matematiska institutionen\\ Stockholms universitet \\
106 91 Stockholm \\
Sweden}
\email{jonathan.rohleder@math.su.se}

\begin{document}

\begin{abstract}
The isoperimetric problem of maximizing all eigenvalues of the Laplacian on a metric tree graph within the class of trees of a given average edge length is studied. It turns out that, up to rescaling, the unique maximizer of the $k$-th positive eigenvalue is the star graph with three edges of lengths $2 k - 1$, $1$ and $1$. This complements the previously known result that the first nonzero eigenvalue is maximized by all equilateral star graphs and indicates that optimizers of isoperimetric problems for higher eigenvalues may be less balanced in their shape---an observation which is known from numerical results on the optimization of higher eigenvalues of Laplacians on Euclidean domains.
\end{abstract}

\maketitle

\section{Introduction}

Within spectral geometry, isoperimetric problems for eigenvalues have a long history that reaches back at least as far as to Lord Rayleigh's famous book {\it The Theory of Sound} \cite[§210]{Rayleigh}. This class of problems deals with finding a shape which maximizes or minimizes (functions of) eigenvalues of the Laplacian or other differential operators under a constraint on a geometric quantity such as the volume, perimeter or diameter of the underlying space. To review just one well-known example, consider the Laplacian with Neumann boundary conditions on a bounded domain $\Omega \subset \R^2$ of area $|\Omega|$ and its eigenvalues $0 = \mu_1 (\Omega) < \mu_2 (\Omega) \leq \mu_3 (\Omega) \leq \dots$. The unique domain $\Omega$ which maximizes the first positive eigenvalue $\mu_2 (\Omega)$ under the constraint $|\Omega| = 1$ is the disc with area one \cite{S54}, while the maximizer of $\mu_3 (\Omega)$ with $|\Omega| = 1$ is the union of two disjoint discs of area $1/2$ each \cite{GNP09}; cf.\ also \cite{BH19}. For higher eigenvalues it is conjectured that the domains maximizing $\mu_4 (\Omega), \mu_5 (\Omega), \dots$ are of less simple shape, cf.\ the numerical observations and pictures in \cite{AF12}. For instance, numerics indicates that the maximizer for $\mu_5 (\Omega)$ is the disjoint union of a ball and a larger, non-convex domain with certain symmetries. For a broad overview on shape optimization problems for eigenvalues of Euclidean domains we refer the reader to~\cite{H06}.

In the present paper we deal with the Laplacian $- \Delta_\Gamma$ on a metric graph $\Gamma$ with standard (continuity--Kirchhoff) vertex conditions, the natural counterpart for metric graphs of the Neumann Laplacian on a domain. We refer to Section~2 for its precise definition. We denote by 
\begin{align*}
 0 = \mu_1 (\Gamma) < \mu_2 (\Gamma) \leq \mu_3 (\Gamma) \leq \dots
\end{align*}
the eigenvalues of $- \Delta_\Gamma$ in nondecreasing order, counted according to their multiplicities. Isoperimetric problems for this operator have come into focus in recent years, where most results deal with $\mu_2 (\Gamma)$, the so-called spectral gap, and its optimizers within the class of graphs with fixed ``volume'' (i.e.\ total length), diameter, or average edge length. Amongst other results it is known by now that $\mu_2 (\Gamma)$ is minimized among all graphs of given total length $L$ by the interval (the graph with two vertices and one edge of length $L$ connecting the two) \cite{N87}, see also \cite{KN14}. If we restrict ourselves to the class of doubly connected graphs, the minimizers were identified to be so-called necklace graphs \cite{BL17}. As simple examples such as equilateral star graphs show, a maximizer of $\mu_2 (\Gamma)$ among graphs of fixed length cannot exist. Instead it turned out that a suitable parameter is the {\em average edge length} 
\begin{align*}
 \cA := \frac{L}{E},
\end{align*}
where $E$ is the number of edges of $\Gamma$ and, again, $L$ is the total length. It was shown in \cite{KKMM16} that the only maximizing graphs of $\mu_2 (\Gamma)$ for fixed $\cA$ are equilateral flower graphs and equilateral pumpkins; see \cite[Section 3]{KKMM16} for their definitions and further details. If one restricts the considered class of graphs to trees, i.e.\ graphs without cycles, then the unique maximizers of $\mu_2 (\Gamma)$ are all equilateral star graphs \cite{R17}. For further related results we refer the reader to \cite{BL17,BKKM17,BKKM19,EJ12,KKT16,K20,KN19,KR20,MP20,P20,RS20}.

While the minimizing result extends to higher eigenvalues, where $\mu_{k + 1} (\Gamma)$ for fixed total length is minimized by the equilateral star with $k + 1$ edges, see \cite{F05}, to the best of our knowledge no results are available yet about which graphs maximize $\mu_{k + 1} (\Gamma)$ for $k \geq 2$ when $\cA$ is fixed. It is the aim of this paper to characterize these maximizers within the class of tree graphs---a class of graphs which seems to share particularly many spectral properties with Euclidean domains. It turns out that the maximizers suffer a certain lack of balance. More precisely, they are non-equilateral and their edge lengths get the more unbalanced the higher $k$ gets. Our main result is the following theorem. 

\begin{theorem}\label{thm:intro}
Let $k \geq 2$. Among all finite, connected metric trees with $E \geq 3$ edges and fixed average length $\cA$, $\mu_{k + 1} (\Gamma)$ is maximal if and only if $\Gamma$ is a star graph with 3 edges of lengths 
\begin{align*}
 \frac{2 k - 1}{2 k + 1} L, \quad \frac{1}{2 k + 1} L, \quad \frac{1}{2 k + 1} L,
\end{align*}
where $L$ denotes the total length of $\Gamma$.
\end{theorem}

The maximizers of the first few eigenvalues are displayed in Figure \ref{fig:intro}. Compared with the known results on $\mu_2 (\Gamma)$ within the class of metric trees with given average length $\cA$, where any equilateral star is a maximizer, it is remarkable that for higher eigenvalues only 3-stars with the specified lengths do the job.
\begin{figure}[htb]
  \centering
  \begin{tikzpicture}
    \draw[fill] (0,0) circle(0.05);
    \draw[fill] (3,0) circle(0.05);
    \draw[fill] (3.87,-0.5) circle(0.05);
    \draw[fill] (3.87,0.5) circle(0.05);
    \draw (0,0)--(3,0);
    \draw (3,0)--(3.87,-0.5);
    \draw (3,0)--(3.87,0.5);
    \node[] at (1.5,-0.4) {$\Gamma_3$};
  \begin{scope}[shift={(5.5,0)}]
    \draw[fill] (0,0) circle(0.05);
    \draw[fill] (5,0) circle(0.05);
    \draw[fill] (5.87,-0.5) circle(0.05);
    \draw[fill] (5.87,0.5) circle(0.05);
    \draw (0,0)--(5,0);
    \draw (5,0)--(5.87,-0.5);
    \draw (5,0)--(5.87,0.5);
    \node[] at (2.5,-0.4) {$\Gamma_4$};
  \end{scope}
  \begin{scope}[shift={(2,-1.5)}]
    \draw[fill] (0,0) circle(0.05);
    \draw[fill] (7,0) circle(0.05);
    \draw[fill] (7.87,-0.5) circle(0.05);
    \draw[fill] (7.87,0.5) circle(0.05);
    \draw (0,0)--(7,0);
    \draw (7,0)--(7.87,-0.5);
    \draw (7,0)--(7.87,0.5);
    \node[] at (3.5,-0.4) {$\Gamma_5$};
  \end{scope}
  \end{tikzpicture}
  \caption{The maximizing trees $\Gamma_j$ of $\mu_j (\Gamma)$ for fixed $\cA$, $j = 3, 4, 5$.}
  \label{fig:intro}
\end{figure}
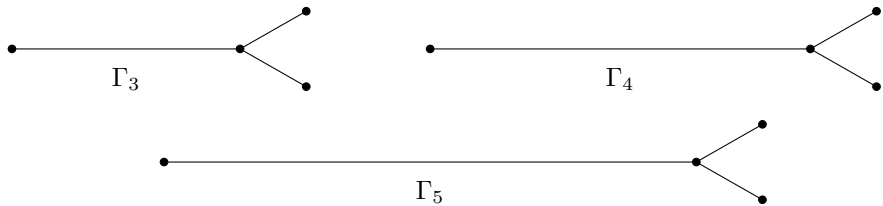

Our proof of Theorem \ref{thm:intro}, based on domain monotonicity properties of metric graph eigenvalues, actually yields an explicit, sharp upper bound for the quantity $\mu_{k + 1} (\Gamma) \cA^2$ depending only on $k$, 
\begin{align*}
 \mu_{k + 1} (\Gamma) \cA^2 \leq \frac{(2 k + 1)^2 \pi^2}{36},
\end{align*}
see Theorem \ref{thm:trees}. This bound may, however, also be obtained from the spectral estimates in \cite{BKKM17}, see Remark \ref{rem:BKKM} below for a more detailed discussion. Thus the present paper confirms the sharpness of the mentioned estimate in \cite{BKKM17} for trees.

\section{Metric graphs, the (standard) Laplacian and its eigenvalues}

A metric graph $\Gamma$ is a discrete graph on a vertex set $\cV$ with edge set $\cE$ that is equipped, additionally, with a length function $L : \cE \to (0, \infty)$. By parametrizing each edge $e$ along the interval $[0, L (e)]$ we may identify $e$ with that interval, and this parametrization induces a natural metric on $\Gamma$. We will always assume that $\Gamma$ is a finite graph, i.e.\ $V := V (\Gamma) := |\cV|$ and $E := E (\Gamma) := |\cE|$ are finite, and that $\Gamma$ is connected. We write $L := L (\Gamma) := \sum_{e \in \cE} L (e)$ for the total length of $\Gamma$. By the finiteness assumption and since we do not allow edges of infinite length, the metric space $\Gamma$ is always compact. We also assume that $\Gamma$ does not contain any loops (i.e.\ edges whose both endpoints correspond to the same vertex). Actually, we will mostly deal with the case that $\Gamma$ is a tree, i.e.\ a graph without cycles, anyway.

By a complex-valued function $f$ on a metric graph $\Gamma$ we mean a collection of functions $f_e : (0, L (e)) \to \C$, $e \in \cE$. In line with this, $f$ belongs to $L^2 (\Gamma)$ if and only if $f_e \in L^2 (0, L (e))$ holds for all $e \in \cE$. Moreover, the Sobolev spaces
\begin{align*}
 \widetilde H^k (\Gamma) := \left\{ f \in L^2 (\Gamma) : f_e \in H^k (\Gamma), e \in \cE \right\}
\end{align*}
for $k = 1, 2, \dots$ and 
\begin{align*}
 H^1 (\Gamma) := \left\{ f \in \widetilde H^1 (\Gamma) : f~\text{is continuous at each vertex} \right\}
\end{align*}
are natural spaces for the treatment of differential operators on metric graphs; in the latter definition, continuity at $v$ means that on all edges incident to the vertex~$v$, $f$ has the same boundary value (or trace) at the endpoint corresponding to $v$. These spaces have the usual properties; for instance they are compactly embedded into~$L^2 (\Gamma)$.

The present text focuses on the Laplacian (i.e.\ the second derivative operator on each edge) subject to standard (also called continuity-Kirchhoff) matching conditions at all vertices. For this, at any vertex $v$, for $f \in \widetilde H^2 (\Gamma)$ we define 
\begin{align*}
 \partial_\nu f (v) : = \sum \partial f_e (v),
\end{align*} 
where the sum is taken over all edges incident to $v$ and $\partial f_e (v)$ denotes the derivative of $f_e$ at the endpoint of $e$ corresponding to $v$, taken in the direction towards $v$. 

\begin{definition}
On any finite, connected metric graph $\Gamma$ the operator $- \Delta_\Gamma$ in $L^2 (\Gamma)$ defined as
\begin{align*}
 (- \Delta_\Gamma f)_e & = - f_e'', \quad e \in \cE, \\
 \dom (- \Delta_\Gamma) & = \left\{ f \in \widetilde H^2 (\Gamma) \cap H^1 (\Gamma) : \partial_\nu f (v) = 0~\text{for all}~v \in \cV \right\},
\end{align*}
is called {\em standard Laplacian} or just {\em Laplacian} on $\Gamma$.
\end{definition}

Note that at the ``loose ends'' (i.e.\ the vertices of degree one) the condition $\partial_\nu f (v) = 0$ simply is a Neumann boundary condition. It is well known that $- \Delta_\Gamma$ is a self-adjoint, non-negative operator. Its lowest eigenvalue is 0 with multiplicity one, with the corresponding eigenfunctions being constant. When ordering the eigenvalues non-decreasingly and counting them with their respective multiplicities (which may be larger than one) we have a sequence
\begin{align*}
 0 = \mu_1 (\Gamma) < \mu_2 (\Gamma) \leq \mu_3 (\Gamma) \leq \dots,
\end{align*}
and we just write $\mu_j$ instead of $\mu_j (\Gamma)$ if the context is clear. In full analogy to the Neumann Laplacian on an interval or a Euclidean domain, the eigenvalues of $-\Delta_\Gamma$ enjoy the variational characterization
\begin{align}\label{eq:minMax}
  \mu_k (\Gamma) = \min_{\substack{F \subset H^1 (\Gamma) \\ \dim F = k}} \max_{\substack{f \in F \\ f \neq 0}} \frac{\int_\Gamma |f'|^2 \dd x}{\int_\Gamma |f|^2 \dd x}, \quad k = 1, 2, \dots.
\end{align}
It is worth mentioning that vertices of degree two do not matter for all our considerations since they may always be added by splitting an edge $e$ into two edges $e', e''$ each of which is incident to the same (new) vertex of degree two and which satisfy $L (e') + L (e'') = L (e)$. This procedure does neither change the domain of $- \Delta_\Gamma$ nor its action, nor, in particular, its eigenvalues.

It has turned out in recent years that eigenvalue inequalities for quantum graphs may often be proven elegantly by using so-called {\em surgery principles}, i.e.\ by employing the (often but not always) monotonous behavior of eigenvalues with respect to surgery operations performed to the metric graph such as adding or removing edges, joining vertices or transplanting subgraphs; see, e.g.,~\cite{BKKM17,BKKM19,KKMM16,KMN13,R17,RS20}. For the proof of the main result of the present note we only need the following surgery principle. It has been well known for long that removing ``pendant'' edges from a graph has a non-decreasing effect on all eigenvalues; see, e.g.,~\cite[Theorem~2]{KMN13} or~\cite[Proposition~3.1]{R17}. However, for us the following necessary condition for equality will be crucial; therefore we provide a short proof.

\begin{proposition}\label{prop}
Let $\Gamma$ be a finite, connected metric graph and let $\Gamma'$ be the graph obtained from $\Gamma$ by removing a pendant edge $e_0$, i.e.\ an edge with a vertex of degree one as one of its endpoints. Then
\begin{align*}
 \mu_{k + 1} (\Gamma) \leq \mu_{k + 1} (\Gamma')
\end{align*} 
holds for $k = 1, 2, \dots$. If $\mu_{k + 1} (\Gamma) = \mu_{k + 1} (\Gamma')$ then there exists an eigenfunction of $- \Delta_{\Gamma'}$ corresponding to the eigenvalue $\mu_{k + 1} (\Gamma')$ which is zero at the vertex $v_0$ of $\Gamma'$ to which $e_0$ is incident in $\Gamma$.
\end{proposition}

\begin{proof}
We interpret $\Gamma'$ as a subset of $\Gamma$. Let $k \in \N$, $\mu := \mu_{k + 1} (\Gamma')$, and let $f_1, \dots, f_{k + 1}$ be pairwise orthogonal eigenfunctions of $- \Delta_{\Gamma'}$ corresponding to the eigenvalues $\mu_1 (\Gamma'), \dots, \mu_{k + 1} (\Gamma')$. Moreover, let $F$ denote the linear span of these functions. An easy integration by parts, taking into account the continuity-Kirchhoff vertex conditions, shows that their derivatives $f_1', \dots, f_{k + 1}'$ are then pairwise orthogonal as well; note that the latter depend on the chosen direction of parametrization of each edge. Extending each function $f \in F$ constantly on $e_0$ in a way that $f$ is continuous on $\Gamma$, we obtain a $(k + 1)$-dimensional subspace $\widetilde F$ of $H^1 (\Gamma)$. If $f \in F \setminus \{0\}$ is orthogonal to $\ker (- \Delta_{\Gamma'} - \mu)$ then 
\begin{align*}
 \frac{\int_\Gamma |\widetilde f'|^2 \dd x}{\int_\Gamma |\widetilde f|^2 \dd x} \leq \frac{\int_{\Gamma'} |f'|^2 \dd x}{\int_{\Gamma'} |f|^2 \dd x} < \mu
\end{align*}
anyway. On the other hand, if each nontrivial $f \in \ker (-\Delta_{\Gamma'} - \mu)$ is nonzero at $v_0$ then for all such $f$
\begin{align*}
 \frac{\int_\Gamma |\widetilde f'|^2 \dd x}{\int_\Gamma |\widetilde f|^2 \dd x} = \frac{\int_{\Gamma'} |f'|^2 \dd x}{\int_{\Gamma'} |f|^2 \dd x + |f (v_0)|^2 L (e_0)} < \frac{\int_{\Gamma'} |f'|^2 \dd x}{\int_{\Gamma'} |f|^2 \dd x} = \mu.
\end{align*}
Hence, in this case, by the min-max principle \eqref{eq:minMax},
\begin{align*}
 \mu_{k + 1} (\Gamma) \leq \max_{\substack{\widetilde f \in \widetilde F \\ \widetilde f \neq 0}} \frac{\int_\Gamma |\widetilde f'|^2 \dd x}{\int_\Gamma |\widetilde f|^2 \dd x} < \mu = \mu_{k + 1} (\Gamma'),
\end{align*}
which proves the proposition.
\end{proof}

We would like to emphasize that the necessary condition for equality given in the previous proposition is not sufficient. In fact, if one adds a sufficiently long (compared to the total length of $\Gamma$) pendant edge to a given metric graph $\Gamma$ then the $k$-th positive eigenvalue will always decrease strictly.

\section{An isoperimetric inequality for higher eigenvalues of the Laplacian}

In this section we state and proof the main result of this article. In fact, the following theorem yields, in particular, the statement of Theorem \ref{thm:intro} in the introduction. Recall that
\begin{align*}
 \cA = \cA (\Gamma) = \frac{L (\Gamma)}{E (\Gamma)}
\end{align*}
denotes the average edge length of $\Gamma$ and that we are assuming throughout that our trees do not contain vertices of degree two; in particular, $\Gamma$ is not a path graph.

\begin{theorem}\label{thm:trees}
Let $\Gamma$ be a finite, connected tree with $E \geq 3$ edges. Then 
\begin{align}\label{eq:bound}
 \mu_{k + 1} \cA^2 \leq \frac{(2 k + 1)^2 \pi^2}{36}
\end{align}
holds for all $k = 1, 2, \dots$. The bound \eqref{eq:bound} is sharp; more precisely, the following assertions hold.
\begin{enumerate}
 \item For $k = 1$, equality holds if and only if $\Gamma$ is any equilateral star graph.
 \item For each $k \geq 2$, equality holds if and only if $\Gamma$ is a 3-star with edge lengths $\frac{2 k - 1}{2 k + 1} L, \frac{1}{2 k + 1} L$, and $\frac{1}{2 k + 1} L$.
\end{enumerate}
\end{theorem}

\begin{remark}\label{rem:BKKM}
We emphasize once more that the bound \eqref{eq:bound} itself can also be derived from \cite[Theorem 4.9]{BKKM17}, which, for the case of trees and standard vertex conditions, reads 
\begin{align}\label{eq:sharp}
 \mu_{k + 1} \leq \left( k - 1 + \frac{|\partial \Gamma|}{2} \right)^2 \frac{\pi^2}{L^2},
\end{align}
where $|\partial \Gamma|$ denotes the number of vertices of degree one. Especially for each 3-star this estimate coincides with the one in Theorem \ref{thm:trees} and, thus, we show that the estimate \eqref{eq:sharp} is sharp for trees. Sharpness of its counterpart for graphs with cycles was earlier established in \cite{KS18}. However, our main interest here is in the class of optimizers of \eqref{eq:bound}, and the following proof shows the bound \eqref{eq:bound} and characterizes all maximizing trees at the same time. 
\end{remark}

\begin{remark}
The bound \eqref{eq:bound} holds also if we admit vertices of degree two, but no optimizing trees may have such vertices. In fact, removing a vertex of degree two (by joining the two incident edges into one edge) does not change the eigenvalues of $- \Delta_\Gamma$, but it strictly increases the average edge length $\cA$. Due to this fact, also the assumption $E \geq 3$ in the theorem is not very restrictive; the only trees which are excluded by this are intervals. However, for the eigenvalues of the Laplacian with standard (Neumann) vertex conditions on an interval we have, by explicit calculation, $\mu_{k + 1} \cA^2 = k^2 \pi^2$, which, by the above theorem, is strictly larger than the value of $\mu_{k + 1} \cA^2$ on any nontrivial metric tree.
\end{remark}

\begin{proof}[Proof of Theorem \ref{thm:trees}]
For $k = 1$ both the estimate and the characterization of maximizers can be found in \cite[Theorem 3.2]{R17}. In the following we show the theorem for $k \geq 2$ in seven steps.

{\bf Step 1:} the estimate \eqref{eq:bound} is true if $\Gamma$ is a 3-star and $k = 2$. That is, on any 3-star $\Gamma$ we have
\begin{align}\label{eq:bound33}
 \mu_3 (\Gamma) \leq \frac{25}{4} \frac{\pi^2}{L (\Gamma)^2}.
\end{align}
To prove this, assume that the edges $e_1, e_2, e_3$ of $\Gamma$ are ordered such that $L (e_1) \geq L (e_2) \geq L (e_3)$. Denote by $\cS$ the equilateral star graph obtained from $\Gamma$ by shortening $e_1$ and $e_2$ to length $L (e_3)$. Then by Proposition \ref{prop}, 
\begin{align*}
 \mu_3 (\Gamma) \leq \mu_3 (\cS) = \frac{9 \pi^2}{4 L (\cS)^2}.
\end{align*}
If we set $\alpha := L (\cS) / L (\Gamma) = 3 L (e_3) / L (\Gamma) \leq 1$, it follows
\begin{align}\label{eq:alpha1}
 \mu_3 (\Gamma) \leq \frac{9 \pi^2}{4 \alpha^2 L (\Gamma)^2}.
\end{align}
On the other hand, if $\Pi$ denotes the path graph formed by $e_1$ and $e_2$ then again
\begin{align}\label{eq:alpha2}
 \mu_3 (\Gamma) \leq \mu_3 (\Pi) = \frac{4 \pi^2}{(L (e_1) + L (e_2))^2} = \frac{4 \pi^2}{(1 - \frac{\alpha}{3})^2 L (\Gamma)^2}.
\end{align}
Now, if $0 < \alpha \leq \frac{3}{5}$ then \eqref{eq:alpha2} yields the bound \eqref{eq:bound33}. On the other hand, the same bound follows from \eqref{eq:alpha1} if $\frac{3}{5} \leq \alpha \leq 1$.

{\bf Step 2:} among all 3-stars, equality in \eqref{eq:bound33} implies that $\Gamma$ has edge lengths $\frac{3}{5} L, \frac{1}{5} L$, and $\frac{1}{5} L$. Indeed, Step 1 of this proof shows that if~$\Gamma$ is a maximizer then $\alpha = 3/5$, i.e., the shortest edge $e_3$ satisfies $L (e_3) = L/ 5$, and at the same time, all inequality signs in the above estimates are equalities. But equality in \eqref{eq:alpha2} implies, by Proposition \ref{prop}, that $e_3$ is attached to the path graph $\Pi$ at a zero of the eigenfunction of $- \Delta_\Pi$ corresponding to $\mu_3 (\Pi)$. Since these zeroes appear at the two symmetric points with distance $L (\Pi) / 4$ to the boundary of $\Pi$, it follows that $L (e_1) = 3 (L (e_1) + L (e_2))/4$. Together with $L (e_3) = L / 5$ this yields that each maximizer $\Gamma$ has the claimed edge lengths. We will see in Step 7 below that 3-stars with the specified edge lengths indeed satisfy the desired equality.

{\bf Step 3:} if $\Gamma$ is any 3-star then the estimate \eqref{eq:bound} holds for all $k$. That is, on any 3-star $\Gamma$ we have
\begin{align}\label{eq:bound3k}
  \mu_{k + 1} (\Gamma) \leq \frac{(2 k + 1)^2}{4} \frac{\pi^2}{L (\Gamma)^2}
\end{align}
for $k = 2, 3, \dots$. We show \eqref{eq:bound3k} by induction over $k$.  For $k = 2$ it was already established in Step 1. Suppose that \eqref{eq:bound3k} holds for some fixed $k$ and each 3-star. Let $\Gamma$ be a 3-star with its edges $e_1, e_2, e_3$ ordered nonincreasingly according to their lengths. Our aim is to show that
\begin{align}\label{eq:inductionBound}
 \mu_{k + 2} (\Gamma) \leq \frac{(2 k + 3)^2}{4} \frac{\pi^2}{L (\Gamma)^2}.
\end{align}
First of all, since $k + 2 \geq 4 = E + 1$, a comparison with the direct sum of the decoupled Neumann Laplacians on the separate edges of $\Gamma$ yields 
\begin{align*}
 \mu_{k + 2} (\Gamma) \geq \frac{\pi^2}{L (e_1)^2}.
\end{align*}
Hence $r := \sqrt{\mu_{k + 2} (\Gamma)}$ satisfies $L (e_1) \geq \pi / r$. Therefore we may consider the graph $\Gamma'$ obtained from $\Gamma$ by removing a piece of length $\pi / r$ from the ``loose end'' of the edge $e_1$. 
If we denote by $\psi_{k + 2}$ an eigenfunction of $- \Delta_\Gamma$ corresponding to $r^2$ then its restriction to $\Gamma'$ will be an eigenfunction of $- \Delta_{\Gamma'}$; in particular, $r^2$ is an eigenvalue on $\Gamma'$ with the same multiplicity as on $\Gamma$,
\begin{align*}
 m := \dim \ker \big(- \Delta_{\Gamma'} - r^2 \big) = \dim \ker \big(- \Delta_\Gamma - r^2 \big).
\end{align*}
Our next aim is to show that
\begin{align}\label{eq:greatDeal}
 r^2 = \mu_j (\Gamma') \quad \text{for some}~j \leq k + 1.
\end{align}
Assume the converse, i.e., $\mu_{k + 1} (\Gamma') < r^2$. If we denote by $I$ the interval of length $\pi/r$ then the disconnected graph consisting of $\Gamma'$ and $I$ as its two connected components has at least $k + 1 + m + 2 = k + m + 3$ eigenvalues in $[0, r^2]$, where we have used that $r^2$ is the second Neumann eigenvalue of $I$. On the other hand, the Laplacian on the disconnected graph is a rank-one perturbation of $- \Delta_\Gamma$ and the latter operator has at most $k + 1 + m$ eigenvalues in $[0, r^2]$, a contradiction. We have proved~\eqref{eq:greatDeal}.\footnote{A slightly more intuitive argument to prove \eqref{eq:greatDeal} goes as follows: generically, the eigenvalue $\mu_{k + 2}$ is simple and its corresponding eigenfunction is a nonzero multiple of $\cos (r x)$ on $e_1$ (assuming $e_1$ is parametrized towards the star vertex) and has exactly $k + 1$ zeroes in $\Gamma$. Cutting away a piece of length $\pi / r$ then leads to an eigenfunction on $\Gamma'$ with exactly $k$ zeroes and, hence, it has to correspond to $\mu_{k + 1} (\Gamma')$. However, the latter argument is less suitable for identifying the maximizers in the next step, since the eigenfunctions of the latter do not satisfy the generic property.}

We may now distinguish two cases. If $L (e_1) = \pi / r$ then $\Gamma'$ is a path graph and
\begin{align*}
 \mu_{k + 2} (\Gamma) = r^2 \leq \mu_{k + 1} (\Gamma') = \frac{k^2 \pi^2}{L (\Gamma')^2} = \frac{k^2 \pi^2}{(L (\Gamma) - \pi/r)^2} < \frac{(2 k + 1)^2}{4} \frac{\pi^2}{(L (\Gamma) - \pi/r)^2}.
\end{align*}
Otherwise, $\Gamma'$ is still a 3-star and, by the induction hypothesis,
\begin{align}\label{eq:bigBoss}
 \mu_{k + 2} (\Gamma) = r^2 \leq \mu_{k + 1} (\Gamma') \leq \frac{(2 k + 1)^2}{4} \frac{\pi^2}{(L (\Gamma) - \pi/r)^2}
\end{align}
as well. Employing this we obtain
\begin{align*}
 r L (\Gamma) - \pi = r (L (\Gamma) - \pi/r) \leq \frac{(2 k + 1) \pi}{2} 
\end{align*}
and thus
\begin{align*}
 r \leq \frac{(2 k + 3) \pi}{2 L (\Gamma)},
\end{align*}
which is \eqref{eq:inductionBound}. 

{\bf Step 4:} among all 3-stars, equality in \eqref{eq:bound3k} implies that the edges have lengths as stated in the theorem. We show this by induction over $k$ again. The case $k = 2$ was treated in Step 2. Suppose that $k \geq 2$ is fixed and that equality holds in \eqref{eq:bound3k} only for the above-stated choice of edge lengths. Assume further that $\Gamma$ is a 3-star for which equality holds in~\eqref{eq:inductionBound}. Then in the reasoning of Step~3 we are in the case that $L (e_1) > \pi / r$ and we must have equality in~\eqref{eq:bigBoss}. But this implies equality in~\eqref{eq:bound3k} with $\Gamma$ replaced by $\Gamma'$, the 3-star obtained from $\Gamma$ by removing a piece of length $\pi/r$ from the loose end of $e_1$. In other words, the 3-star $\Gamma'$ maximizes $\mu_{k + 1} \cA^2$, and from the induction assumption we obtain that $\Gamma'$ has edge lengths $L' (e_1) = \frac{2 k - 1}{2 k + 1} L (\Gamma'), L' (e_2) = L' (e_3) = \frac{1}{2 k + 1} L (\Gamma')$. By construction, the edge lengths of $\Gamma$ are then given by
\begin{align*}
 L (e_1) & = L' (e_1) + \pi/r = \frac{2 k - 1}{2 k + 1} \left( L (\Gamma) - \frac{2 L (\Gamma)}{2 k + 3} \right) + \frac{2 L (\Gamma)}{2 k + 3} = \frac{2 k + 1}{2 k + 3} L (\Gamma)
\end{align*}
and, for $j = 2, 3$,
\begin{align*}
 L (e_j) = L' (e_j) = \frac{1}{2 k + 1} \left( L (\Gamma) - \frac{2 L (\Gamma)}{2 k + 3} \right) = \frac{1}{2 k + 3} L (\Gamma).
\end{align*}
Hence among 3-stars any maximizers need to have the lengths specified in the theorem.

{\bf Step 5:} the estimate \eqref{eq:bound} holds on arbitrary trees. For this let now $\Gamma$ be an arbitrary finite, connected tree with $E \geq 3$. Let $e_1, e_2, e_3$ be three edges such that $L (e_1) \geq L (e_2) \geq L (e_3) \geq L (e)$ for all $e \in \cE$, $e \neq e_1, e_2, e_3$. Within $\Gamma$ choose any maximal path $\Pi_1$ which contains $e_1$ and $e_2$ and connects two vertices of degree one. Let $\hat e_3$ be such that $\hat e_3$ is not contained in $\Pi_1$ but has maximal length in $\Gamma \setminus \Pi_1$, i.e., $L (\hat e_3) \geq e$ for all edges $e$ not belonging to $\Pi_1$; if $e_3$ is not part of $\Pi_1$ then we choose $\hat e_3 = e_3$. Furthermore, let $\Pi_2$ denote any path which contains $\hat e_3$ and connects a vertex of degree one with a vertex on $\Pi_1$ without having any joint edge with $\Pi_1$. Then $\cS := \Pi_1 \cup \Pi_2$ is a connected subgraph of $\Gamma$ and it may, after cutting off all further edges of $\Gamma$ and then removing all vertices of degree two, be viewed as a $3$-star. Moreover, by construction the longest edges $e_1, e_2, e_3$ of $\Gamma$ are contained in $\cS$ and hence
\begin{align}\label{eq:avLengths}
 \cA (\cS) = \frac{L (\cS)}{3} \geq \frac{L (e_1) + L (e_2) + L (e_3)}{3} \geq \frac{L (\Gamma)}{E (\Gamma)}.
\end{align}
Consequently, by the result of Step 3 and Proposition \ref{prop},
\begin{align}\label{eq:fetzt}
 \mu_{k + 1} (\Gamma) \leq \mu_{k + 1} (\cS) \leq \frac{(2 k + 1)^2}{36} \frac{E (\cS)^2 \pi^2}{L (\cS)^2} \leq \frac{(2 k + 1)^2}{36} \frac{E (\Gamma)^2 \pi^2}{L (\Gamma)^2},
\end{align}
which proves \eqref{eq:bound}.

{\bf Step 6:} equality in \eqref{eq:bound} implies that $\Gamma$ is a 3-star and has the edge lengths specified in the theorem. To this end, let us assume that $\Gamma$ is a tree for which equality holds in \eqref{eq:bound} for some $k \geq 2$. It suffices to show that $\Gamma$ is a 3-star; after that the lengths property follows from Step 4. First of all, from the equality in \eqref{eq:bound} we get, in particular, equalities everywhere in \eqref{eq:avLengths} and \eqref{eq:fetzt}. The first (in)equality in~\eqref{eq:avLengths} then implies $L (e_1) + L (e_2) + L (e_3) = L (S)$, so that $S$ consists only of $e_1, e_2$ and $e_3$; by the construction of $S$ this also yields that $e_1, e_2$ and $e_3$ all are incident to vertices of degree one in $\Gamma$. Hence $\cS$ is the 3-star with $e_1, e_2, e_3$ as its edges, and from the equalities in \eqref{eq:fetzt} we get, furthermore, that $\cS$ is a maximizer itself and, by Step 4, has to have edge lengths
\begin{align}\label{eq:yeah}
 L (e_1) = \frac{2 k - 1}{2 k + 1} L (\cS), \quad L (e_2) = L (e_3) = \frac{1}{2 k + 1} L (\cS).
\end{align}
It remains to show that $\Gamma = \cS$. In fact, any edge $e$ different from $e_1, e_2, e_3$ necessarily would have to satisfy $L (e) \leq \frac{1}{2 k + 1} L (\cS)$, and since $L (e_1)$ is larger, this would yield $L (\Gamma)/E (\Gamma) < L (\cS) / E (\cS)$, in contradiction to the equality in \eqref{eq:avLengths}. Therefore $\Gamma = \cS$ and \eqref{eq:yeah} is the desired statement on the edge lengths.

{\bf Step 7:} each 3-star with edge lengths $\frac{2 k - 1}{2 k + 1} L$, $\frac{1}{2 k + 1} L$, $\frac{1}{2 k + 1} L$ satisfies 
\begin{align}\label{eq:dasGrosseGanze}
 \mu_{k + 1} = \mu_{k + 2} = \frac{(2 k + 1)^2 \pi^2}{4 L^2}
\end{align}
and, thus, yields equality in \eqref{eq:bound}. 

Firstly, note that the expression in \eqref{eq:dasGrosseGanze} is indeed an eigenvalue of multiplicity two since the function $\cos ( (2 k + 1) \pi x /(2 L))$ along the path consisting of the long and one of the short edges and complemented by zero on the other short edge is an eigenfunction and we may interchange the roles of the two short (and equally long) edges.

Secondly, by the estimate \eqref{eq:bound} proven in Step 1 and 3 we have
\begin{align}\label{eq:erfolg}
 \mu_k \leq \frac{(2 k - 1)^2 \pi^2}{4 L^2} < \frac{(2 k + 1)^2 \pi^2}{4 L^2},
\end{align}
and, on the other hand, the disconnected graph $\cD$ consisting of a path formed by $e_1$ and $e_2$ as one connected component and the separated edge $e_3$ as its other connected component satsifies, by an easy computation,
\begin{align*}
 \mu_{k + 2} (\cD) = \frac{(2 k + 1)^2 \pi^2}{4 L^2} < \frac{(k + 1)^2 (2 k + 1)^2 \pi^2}{4 k^2 L^2} = \mu_{k + 3} (\cD) \leq \mu_{k + 3} (\Gamma),
\end{align*}
the latter inequality being valid as $- \Delta_\cD$ is a rank-one perturbation of $- \Delta_\Gamma$. Together with \eqref{eq:erfolg} this yields \eqref{eq:dasGrosseGanze} as its only possible conclusion and completes this proof.
\end{proof}

\begin{remark}
In Step 7 of the proof of the previous theorem we have seen more than the theorem claims. Indeed, for the 3-stars which maximize $\mu_{k + 1}$ we always have $\mu_{k + 1} = \mu_{k + 2}$. This is in line with other results on isoperimetric inequalities, where for the optimally shaped domains and graphs the eigenvalues in question are often multiple.
\end{remark}

\begin{ack}
The author is grateful to the Swedish Research Council (VR) for supporting this research financially through grant no.\ 2018-04560.
\end{ack}


\begin{thebibliography}{99}

\bibitem{AF12} P.\ R.\ S.\ Antunes and P.\ Freitas, {\it Numerical optimization of low eigenvalues of the Dirichlet and Neumann Laplacians}, J.\ Optim.\ Theory Appl.\ 154 (2012), 235--257.


\bibitem{BL17} R.~Band and G.~L\'evy, {\it Quantum graphs which optimize the spectral gap}, Ann.\ Henri Poincar\'e~18 (2017), 3269--3323.



\bibitem{BKKM17} G.\ Berkolaiko, J.\,B.\ Kennedy, P.\ Kurasov, and D.\ Mugnolo, {\it Edge connectivity and the spectral gap of combinatorial and quantum graphs}, J.\ Phys.\ A 50 (2017), 365201.

\bibitem{BKKM19} G.\ Berkolaiko, J.\,B.\ Kennedy, P.\ Kurasov, and D.\ Mugnolo, {\it Surgery principles for the spectral analysis of quantum graphs}, Trans.\ Amer.\ Math.\ Soc.~372 (2019), 5153--5197. 

%

%

\bibitem{BH19} D.\ Bucur and A.\ Henrot, {\it Maximization of the second non-trivial Neumann eigenvalue}, Acta Math.\ 222 (2019), 337--361.

%
%
%

%
\bibitem{EJ12} P.~Exner and M.~Jex, {\it On the ground state of quantum graphs with attractive $\delta$-coupling}, Phys.\ Lett.~A~376 (2012), 713--717.



%
%
\bibitem{F05} L.~Friedlander, {\it Extremal properties of eigenvalues for a metric graph}, Ann.\ Inst.\ Fourier~55 (2005), 199--211.

\bibitem{GNP09} A.\ Girouard, N.\ Nadirashvili, and I.\ Polterovich, {\it Maximization of the second positive Neumann eigen-value for planar domains} J.\ Differ.\ Geom.\ 83 (2009), 637--662.

\bibitem{H06} A.\ Henrot, Extremum Problems for Eigenvalues of Elliptic Operators, Frontiers in Mathematics, Birkh\"auser Verlag, Basel, 2006.

\bibitem{KKT16} G.~Karreskog, P.~Kurasov, and I.~Trygg Kupersmidt, {\it Schr\"odinger operators on graphs: Symmetrization and Eulerian cycles}, Proc.\ Amer.\ Math.\ Soc.~144 (2016), 1197--1207.

\bibitem{K20} J.~Kennedy, {\it A sharp eigenvalue bound for quantum graphs in terms of their diameter}, to appear in Oper.\ Theory Adv.\ Appl.\ 281.
%
\bibitem{KKMM16} J.\,B.~Kennedy, P.~Kurasov, G.~Malenov\'a, and D.~Mugnolo, {\it On the spectral gap of a quantum graph}, Ann. Henri Poincar\'e 17 (2016), 2439--2473.

\bibitem{KN19} A.\ Kostenko and N.\ Nicolussi, {\it Spectral estimates for infinite quantum graphs}, Calc.\ Var.\ Partial Differential Equations 58 (2019), Paper No.\ 15.


%
%
%
 \bibitem{KMN13} P.~Kurasov, G.~Malenov\'a, and S.~Naboko, {\it Spectral gap for quantum graphs and their connectivity}, J.\ Phys.\ A~46 (2013), 275309.
%
\bibitem{KN14} P.~Kurasov and S.~Naboko, {\it Rayleigh estimates for differential operators on graphs}, J.\ Spectral Theory~4 (2014), 211--219.

\bibitem{KR20} P.\ Kurasov and J.\ Rohleder, {\it Laplacians on bipartite metric graphs}, to appear in Oper.\ Matrices.

\bibitem{KS18} P.\ Kurasov and A.\ Serio, {\it On the sharpness of spectral estimates for graph Laplacians} Rep.\ Math.\ Phys.\ 82 (2018), 63--80.
%

\bibitem{MP20} D.\ Mugnolo and M.\ Pl\"umer, {\it Lower Estimates on Eigenvalues of Quantum Graphs}, preprint, arXiv:1907.13350.
%
%
\bibitem{N87} S.~Nicaise, {\it Spectre des r\'esaux topologiques finis}, Bull.\ Sci.\ Math.~111 (1987), 401--413.
%

\bibitem{P20} M.\ Pl\"umer, {\it Upper Eigenvalue Bounds for the Kirchhoff Laplacian on Embbeded Metric Graphs}, preprint, arXiv:2004.03230.

\bibitem{Rayleigh} J.\,W.\,S.~Rayleigh, The Theory of Sound, Macmillan, London, 1877, 1st edition (reprinted: Dover, New York, 1945).
%
\bibitem{R17} J.~Rohleder, {\it Eigenvalue estimates for the Laplacian on a metric tree}, Proc.\ Amer.\ Math.\ Soc.~145 (2017), 2119--2129.

\bibitem{RS20} J.~Rohleder and C.~Seifert, {\it Spectral monotonicity for Schr\"odinger operators on metric graphs}, to appear in Oper.\ Theory Adv.\ Appl.\ 281.

\bibitem{S54} G. Szeg\H{o}, {\it Inequalities for certain eigenvalues of a membrane of given area}, J.\ Rational Mech.\ Anal.\ 3 (1954), 343--356.


\end{thebibliography}
\end{document}